\documentclass[11pt,final]{article}   % add "final" option to remove showkeys
\usepackage{a4wide}
\usepackage{amsmath}
\usepackage{amssymb}
\usepackage{mathrsfs}
\usepackage{amsmath}
\usepackage{amssymb}
\usepackage{amsthm,xcolor}
\usepackage{enumerate,enumitem}
\usepackage[margin=3cm]{geometry}

\usepackage[colorlinks,linkcolor=black,citecolor=black]{hyperref} 
\usepackage[bordercolor=orange,backgroundcolor=orange!20,linecolor=orange,textsize=scriptsize,obeyFinal]{todonotes}
\usepackage[notcite,notref]{showkeys}  % remove color option if conflicts with tikz
\usepackage{cleveref}

\newtheorem{theorem}{Theorem}%[section]
\newtheorem{definition}[theorem]{Definition}

\newtheorem{observation}[theorem]{Observation}
\newtheorem{lemma}[theorem]{Lemma}

\newtheorem{conjecture}[theorem]{Conjecture}
\newtheorem{question}[theorem]{Question}

\def\subset{\subseteq}
\def\se{\subseteq}
\def\sm{\setminus}
\newcommand{\im}{\mathrm{Im\,}}
\newcommand{\dcup}{\dot\cup}

\usepackage[square,sort,comma,numbers]{natbib} 
\usepackage{setspace}
\setlist{itemsep=2pt,parsep=1pt,topsep=3pt,partopsep=0pt}  
\setenumerate{leftmargin=*,labelindent=\parindent} 

\def\itm#1{\rm ({#1})} 
\def\itmit#1{\itm{\it #1\,}} 
 
\def\abc{\itmit{\alph{*}}}
 
\def\itmarab#1{\mbox{\itm{{\it #1\,}\arabic{*}\hspace{.05em}}}}

\newcommand{\By}[2]{\overset{\mbox{\tiny{#1}}}{#2}}
\newcommand{\ByRef}[2]{   \By{\eqref{#1}}{#2} }

\newcommand{\eqByRef}[1]{ \ByRef{#1}{=} }

\newcommand{\geByRef}[1]{ \ByRef{#1}{\ge} }

\newcommand{\cB}{\ensuremath{\mathcal{B}}}

\newcommand{\cH}{\ensuremath{\mathcal{H}}}
\newcommand{\E}{\ensuremath{\mathbb{E}}} 	% Expectation
\newcommand{\wv}{\ensuremath{\boldsymbol{w}}} 
\newcommand{\cE}{\mathcal{E}}
\newcommand{\eps}{\ensuremath{\varepsilon}} 

\author{Peter Allen\thanks{Department of Mathematics, London School of Economics, Houghton Street, London WC2A 2AE, UK. Email: {\tt p.d.allen@lse.ac.uk}} \and Julia B\"ottcher\thanks{Department of Mathematics, London School of Economics, Houghton Street, London WC2A 2AE, UK. Email: {\tt j.boettcher@lse.ac.uk}} \and 
Anita Liebenau\thanks{School of Mathematics and Statistics, UNSW Sydney, NSW 2052, Australia. Email: {\tt a.liebenau@unsw.edu.au.} Supported by the Australian research council.}}
\date{}

\begin{document}

\title{Universality for graphs of bounded degeneracy}

\maketitle

\begin{abstract}
Given a family $\cH$ of graphs, a graph $G$ is called $\cH$-universal if $G$ contains every  graph of $\cH$ as a subgraph. Following the extensive research on universal graphs of small size for bounded-degree graphs, Alon asked what is the minimum number of edges that a graph must have to be universal for the class of all $n$-vertex graphs that are $D$-degenerate. In this paper, we answer this question up to a factor that is polylogarithmic in $n.$
\end{abstract}

\section{Introduction}
Given a family $\cH$ of graphs, a graph $G$ is called {\em $\cH$-universal} if $G$ contains a copy of $H$ as a subgraph for every $H\in \cH.$ Rado~\cite{rado1964} first investigated universal graphs for infinite graphs. Since then, a lot of research has focused on finding sparse graphs that are universal, for various finite classes of graphs. Constructions of universal graphs with few edges have practical applications, for example, in space-efficient VLSI circuit design~\cite{valiant1981}, in data representation~\cite{crs1983,rss1980}, and in parallel computing~\cite{bl1982,bclr1986}. 

Specific classes of $\cH$ for which $\min \{ e(G) \mid G \text{ is } \cH\text{-universal}\}$ was studied include the class of all graphs with exactly $m$ edges~\cite{bcegs1982,aa2002,bds2021}, 
forests~\cite{cg1978,cg1983,cgp1976,fp1987,bclr1989}, planar graphs~\cite{bcegs1982,bclr1989,c2002}, and graphs of bounded maximum degree~\cite{aa2002,ck1999,ac2007,ac2008,ackrrs2000,ackrrs2001}. Let $\cH_{\Delta}(n)$ denote the family of all graphs on $n$ vertices that have maximum degree at most $\Delta$. 
Building on earlier work, Alon and Capalbo~\cite{ac2008} constructed an $\cH_{\Delta}(n)$-universal graph with $O(n^{2-{2/\Delta}})$ edges, which is tight up to the multiplicative constant by a counting argument due to Alon, Capalbo, Kohayakawa, R\"odl, Ruci\'nski and Szemer\'edi~\cite{ackrrs2000}.

We mention here that most of the constructions in the above references, including~\cite{ac2008}, are explicit. In an earlier paper, however, Alon, Capalbo, Kohayakawa, R\"odl, Ruci\'nski and Szemer\'edi~\cite{ackrrs2000}  showed that the binomial random graph $G(n,p)$ is asymptotically almost surely (a.a.s) $\cH_{\Delta}((1-\eps)n)$-universal (also called {\em almost-spanning} universal), when the edge probability $p$ is at least of order $(\log n/n)^{1/\Delta}.$ 
Since then, the problem of finding a threshold function for the random graph $G(n,p)$ to be $\cH_{\Delta}((1-\eps)n)$-universal or $\cH_{\Delta}(n)$-universal (i.e.~{\em spanning universal}), for given $\Delta$, has received a lot of attention and is still widely open. 
A threshold is at least of order $n^{2/(\Delta +1)}$ for almost-spanning universality, see e.g.~\cite{cfns2017}, and for spanning universality at least of order $n^{2/(\Delta +1)}(\log n)^{1/\binom{\Delta+1}{2}}$, due to the celebrated Johannson-Kahn-Vu theorem~\cite{jkv2008} on the threshold probability for clique factors. 
For general $\Delta\ge 3$ the best known upper bounds are 
$O(n^{1/(\Delta-1)}\log^5 n)$ for almost-spanning universality due to Conlon, Ferber, Nenadov and \v{S}kori\'c~\cite{cfns2017};  
and $O (n^{-1/(\Delta-0.5)}\log^3n)$ for spanning universality due to Ferber and Nenadov~\cite{fn2018}, beating the long standing barrier of $(\log n/n)^{1/\Delta}.$ 
Note that the result in~\cite{cfns2017} is tight for $\Delta=3$ up to the polylogarithmic term. 
In a recent breakthrough, Ferber, Kronenberg and Luh~\cite{fkl2019} proved that the Johannson-Kahn-Vu lower bound yields the correct order of magnitude for $G(n,p)$ to be (spanning) $\cH_2(n)$-universal. A folklore conjecture, stated explicitly in~\cite{fkl2019}, says that this ought to be true for all $\Delta\ge 3.$

Interestingly, some of these proofs actually give a better bound in terms of the degeneracy (when the maximum degree is still bounded but possibly much larger). 
A graph $H$ is said to have degeneracy $D$ if every induced subgraph of $H$ has a vertex of degree at most $D$. Equivalently, $H$ is $D$-degenerate if there is an ordering $v_1,\ldots,v_n$ of the vertices of $H$ such that $v_i$ has at most $D$ neighbours in $H$ among the vertices $\{v_1,\ldots,v_{i-1}\}.$  We denote by $\cH(n,D)$ the family of all $n$-vertex graphs of degeneracy at most $D$, and by $\cH_{\Delta}(n,D)$ the subfamily of graphs that additionally have maximum degree at most $\Delta.$ 

Ferber and Nenadov~\cite{fn2018} prove, as a simple example of their new ideas, that $G(n,p)$ is a.a.s.~$\cH_{\Delta}(n,D)$-universal for $p\ge (n^{-1}\log^3 n)^{1/2D}$, i.e., contains every $n$-vertex $D$-degenerate graph of maximum degree $\Delta$. That is, for graphs of  degeneracy much smaller than the maximum degree, the exponent $1/2D$ is much better than the general bound. An even simpler proof for  almost-spanning universality is included by Nenadov in his thesis~\cite{rajko-thesis}, where he proves that for some $p = O(\log^2n/(n\log\log n))^{1/D}$, the random graph $G(n,p)$ is $\cH_{\Delta}((1-\eps)n,D)$-universal.

The following question of Alon~\cite{a-personal} asks for universal graphs for graphs of bounded degeneracy, but arbitrarily large maximum degree.  
\begin{question}[Alon]\label{q:Alon}
What is $\min\{e(H): H \textit{ is universal for  } \mathcal{H}(n,D) \}$?
\end{question}

Observe that $G(n,p)$ is not a suitable candidate for such a universal graph as its maximum degree is only $O(pn)$. But, as we shall show, a random block model does work.

In this paper, we answer Question~\ref{q:Alon} up to a poly-logarithmic factor. First, we adapt the lower bound proof of Alon, Capalbo, Kohayakawa, R\"odl, Ruci\'nski and Szemer\'edi~\cite{ackrrs2000} to a similar argument for $D$-degenerate graphs. 
\begin{theorem}\label{thm:lowDens}
Given $D\ge1$, suppose that $n$ is sufficiently large and that the graph $\Gamma$ contains all $D$-degenerate graphs on $n$ vertices with maximum degree at most $2D+1$. Then $e(\Gamma)\ge\tfrac1{1000D}n^{2-1/D}$.
\end{theorem}
This result shows that Nenadov's upper bound~\cite{rajko-thesis} on universality for $D$-degenerate graphs whose maximum degree is in addition bounded is tight up to a poly-logarithmic factor. This is complemented by our main theorem, which shows that $n^{2-1/D}$ is tight up to a poly-logarithmic factor even without a maximum degree restriction. 
\begin{theorem}\label{thm:Construction}
Given $D\ge1$ and $n$ sufficiently large. Then there exists a graph with at most
\[80000n^{2-1/D}(\log^{2/D} n)(\log\log n)^5\,.\]
 edges that is $\cH(n,D)$-universal. 
\end{theorem}

\section{Proofs}

We first prove \Cref{thm:lowDens} which, similarly to the lower bound of $n^{2-2/\Delta}$ for $\cH_{\Delta}(n)$-universal graphs in~\cite{ackrrs2000}, follows from a counting argument. 

\begin{proof}[Proof of \Cref{thm:lowDens}]
 Observe that since $\Gamma$ contains all graphs of maximum degree $D$, it has $\Omega\big(n^{2-1/D}\big)>3Dn$ edges by~\cite{ackrrs2000}.

 We first count connected graphs on $[n]$ with maximum degree at most $2D+1$ such that the natural order on $[n]$ is a $D$-degeneracy order. We can construct any such graph as follows: for each $2\le i\le n$ in succession, we pick between $1$ and $D$ vertices coming before $i$ whose degree is currently $2D$ or smaller, and join $i$ to these vertices. Note that since the vertex $i-1$ has at most $D$ neighbours, there is always at least one vertex to choose.
 
 Consider the choices for neighbours of vertex $i$ in this process. We have at most $D(i-1)$ edges within $[i-1]$. Letting $s$ be the number of vertices of degree $2D+1$ at this point, we see that
 \[\tfrac12(2D+1)s\le D(i-1)\quad\text{and so}\quad s\le \tfrac{2D}{2D+1}(i-1)\,.\]
 In particular, the number of vertices with $2D$ or fewer neighbours in $[i-1]$ is at least $\tfrac{1}{2D+1}(i-1)$, which is at least $\tfrac1{4D}i$ when $n$ is sufficiently large and $i\ge\sqrt{n}$. Thus, in this case, the number of ways to choose edges at $i$ is at least $\binom{i/4D}{D}\ge 10^{-D}D^{-2D}i^D$. Multiplying, the total number of choices in this process is at least
 \[\prod_{i=\sqrt{n}}^n10^{-D}D^{-2D}i^D\ge 100^{-Dn}D^{-2Dn}n^{Dn}\,.\] 
 This is therefore a lower bound on the number of connected graphs on $[n]$ with maximum degree $2D+1$ such that the natural order is a $D$-degeneracy order.
 
 We now consider how many such graphs can appear in $\Gamma$. For any given $n-1\le q\le Dn$, we use the following procedure. We pick some $q$ edges of $\Gamma$. If these edges span exactly $n$ vertices, we pick a labelling of the $n$ vertices with $[n]$.
 
 Since $e(\Gamma)>3Dn$, the number of graphs on $[n]$ we obtain like this is at most
 \[\sum_{q=n-1}^{Dn}\binom{e(\Gamma)}{q}n!\le 2\binom{e(\Gamma)}{Dn}n!\le 10^{Dn}\Big(\tfrac{e(\Gamma)}{Dn}\Big)^{Dn}n^n\,.\]
 
 If $\Gamma$ contains all connected $n$-vertex $D$-degenerate graphs with maximum degree at most $2D+1$, then in particular all the graphs on $[n]$ we constructed are obtained by this procedure, so we have
 \[100^{-Dn}D^{-2Dn}n^{Dn}\le 10^{Dn}\Big(\tfrac{e(\Gamma)}{Dn}\Big)^{Dn}n^n\]
 and hence
 \[e(\Gamma)^{D}\ge 100^{-D}D^{-D}n^{2D-1}10^{-D}=(1000D)^{-D} n^{2D-1}\,, \]
 from which the theorem follows. 
\end{proof}

We now turn to the proof of \Cref{thm:Construction}. We will first explain
the randomised construction we use for our universal graph, then provide the details of our
embedding strategy and formulate a pseudo-randomness property our construction
has, which we can then use to prove that our strategy works. We shall apply
the following Chernoff bound.

\begin{theorem}[Chernoff bound~{\cite[Theorem~2.1]{JLR}}]\label{thm:Chernoff}
  Let $X$ be a binomial random variable. Then for $\delta\in (0,3/2),$ we have 
  $\Pr( |X-\E X| > \delta\;\E X) < 2 e^{-\delta^2\E X/3}. $
\end{theorem}

We aim to construct a graph $\Gamma$ with the desired number of edges that
contains every $D$-degenerate graph $G$ on $n$ vertices. Clearly, a
$D$-degenerate graph on $n$ vertices may contain vertices of degree up to $n-1$,
and hence taking $\Gamma$ to be a standard binomial random graph will not work:
If we choose the edge density~$p$ appropriately for the number of desired edges
in~$\Gamma$, we will asymptotically almost surely not be able to embed vertices
of degree larger than $2np = o(n)$. However, by counting edges, we easily
observe that a $D$-degenerate graph cannot contain too many vertices of large
degree.

\begin{observation}\label{degree-observation}
  If $G$ is a $D$-degenerate graph then the number of vertices in $G$ of degree
  at least $k$ is at most $2Dn/k.$
\end{observation}

With this in mind, the vertex set of our host graph $\Gamma$ will consist of
pairwise disjoint sets $W_{1},\ldots, W_{N}$, where for $1\le i\le N$ the host
set $W_i$ will be used for embedding vertices of degree between $n^{D^{-i}}$ and
$n^{D^{1-i}}$, and hence $W_{i}$ has size $\Theta(n^{1-D^{-i}})$. We then add
random edges between and within these sets with edge probabilities chosen so
that we can embed vertices of the desired degrees. Here, the parameters are
chosen so that we also obtain the correct overall number of edges (as we show
below). The following definition formalises this construction.

\begin{definition}[random block model]\label{construction}
  Given~$n$ and~$D$, let $N$ be the smallest integer such that
  \begin{equation}\label{eq:N}
    n^{D^{1-N}}\le
    3^{D^2}\,.
  \end{equation}
  For each $1\le i,k\le N$, let
  \begin{equation}\label{eq:pDelta}
    p_{i,k}=\min\big\{n^{-D^{-1}+D^{-i}+D^{-k}}(\log^{2/D} n)(\log\log n)^3,1\big\}\,,
    \quad\text{and}\quad
    \Delta_i=n^{D^{1-i}}\,.
  \end{equation}
  The \emph{random block model} $\Gamma(n,D)$ then has vertex set
  $W=W_{1}\dcup\ldots\dcup W_{N}$, where the pairwise disjoint $W_k$ are
  called \emph{blocks} and are of size
  \[|W_k|=100\cdot 3^Dn^{1-D^{-k}}\,.\]
  The edge set of the
  random block model is obtained as follows.  For each pair of vertices $u,v$
  with $u\in W_i$ and $v\in W_k$ we let $uv$ be an edge independently with
  probability $p_{i,k}$.
  
  For our embedding strategy it turns out to be useful to further partition
  each block $W_k$ into a \emph{sub-block} $W_{k,1}$ of size $\tfrac12|W_k|$ and
  sub-blocks $W_{k,2},\dots,W_{k,\log n}$ all of size at least $\tfrac1{2\log
    n}|W_k|$.
\end{definition} 

Our model has the following easy properties; in particular $\Gamma(n,D)$ has the
correct number of edges for our purposes a.a.s.

\begin{lemma}[properties of the block model]\label{lem:blockmodel}
  For sufficiently large~$n$ and $D\ge 2$, for $N,\Delta_N$, the random block model $\Gamma(n,D)$ with probabilities $p_{i,k}$ and
  blocks $W_k$ as in \Cref{construction} the following hold.
  \begin{enumerate}[label=\abc]
  \item\label{lem:blockmodel:N} $\frac{\log\log n}{2\log D}\le N\le 2\log\log n$ and $3^D\le\Delta_N\le 3^{D^2}$,
  \item\label{lem:blockmodel:p} if $i=1$ or $k=1$ we have $p_{i,k}=1$, otherwise $p_{i,k}=n^{-D^{-1}+D^{-i}+D^{-k}}(\log^{2/D} n)(\log\log n)^3$,
  \item\label{lem:blockmodel:W} $100n\le |W_N|\le \frac{100}{3}\cdot 3^D n$,
  \item\label{lem:blockmodel:Gamma} $\Gamma(n,D)$ has at most $200\cdot 3^D n$ vertices, and a.a.s.~at most 
    \[10^5\cdot 3^{2D} n^{2-1/D}(\log^{2/D} n)(\log\log n)^5\]
    edges. 
  \end{enumerate}
  \end{lemma}
\begin{proof}  
  Properties~\ref{lem:blockmodel:N} and~\ref{lem:blockmodel:p} are
  straightforward implications of the definitions, where for the estimates on
  $\Delta_N$ we use $\Delta_N=n^{D^{1-N}}\le 3^{D^2}$ and
  $\Delta_N^D=n^{D^{1-(N-1)}}>3^{D^2}$ by the definitions of~$N$ and~$\Delta_N$.
  We obtain~\ref{lem:blockmodel:W} by observing that
  \[|W_N|=100\cdot 3^Dn^{1-D^{-N}}=100\cdot 3^Dn\frac{1}{(\Delta_N)^{1/D}}\]
  and plugging in $3^D\le\Delta_N\le 3^{D^2}$. The first part of~\ref{lem:blockmodel:Gamma}
  follows from
  \[\sum_{k=1}^N|W_k|=100\cdot 3^Dn\sum_{k=1}^N\Big(\frac{|W_N|}{100\cdot 3^Dn}\Big)^{D^k}\]
  and $|W_N|\le \frac{100}{3}\cdot 3^D n$. It remains to prove the second part of~\ref{lem:blockmodel:Gamma}, which
  is an easy application of a Chernoff bound.
  
  Indeed, for every pair $(i,k)\in[N]^2$ let $E(W_i,W_k)$ denote the edges
  between $W_i$ and $W_k$ in $\Gamma.$
  We have $\E |E(W_i,W_k)| = p_{i,k}|W_i||W_k|$ if $i\neq k$, and
  $\E |E(W_k,W_k)| = p_{k,k}
  \binom{|W_k|}{2}=(\frac38\pm\frac18)p_{k,k}|W_k||W_k|$.
  If $i,k\neq 1$ we have
  $p_{i,k}|W_i||W_k|=\log^{2/D}n(\log\log n)^3(100\cdot 3^D)^2n^{2-D^{-1}}$,
  and if $i=1$ (and similarly for $k=1$) we have
  $p_{i,k}|W_i||W_k|=(100\cdot 3^D)^2n^{2-D^{-1}-D^{-k}}$.
  We conclude that for each pair $(i,k)\in[N]^2$ and $n$ large enough,
  \begin{multline*}
    n^{2-2/D} \le \E |E(W_1,W_1)| \le \E |E(W_i,W_k)| \\
    \le p_{i,k}|W_i||W_k|\le 10^4\cdot3^{2D}n^{2-1/D}(\log^{2/D} n)(\log\log n)^3\,.
  \end{multline*}
  Thus, by the Chernoff bound in \Cref{thm:Chernoff} and the union bound over the $N^2\le (2\log\log
  n)^2$ pairs, with probability at most $2 N^2 e^{-(n^{2-2/D})/3} = o(1)$,
  the total number of edges in $\Gamma(n,D)$ exceeds
  \[2\cdot (2\log\log n)^2\cdot 10^4\cdot3^{2D}n^{2-1/D}(\log^{2/D} n)(\log\log n)^3,\]  
  and the claim follows. 
\end{proof}

In the proof of \Cref{thm:Construction} we will show that $\Gamma(n,D)$
a.a.s.~contains every $D$-degenerate graph on $n$ vertices.  For this we shall
use the following embedding strategy.

\begin{definition}[Embedding Strategy]\label{strategy}
  Given $\Gamma\sim\Gamma(n,D)$, fix a $D$-degenerate graph $H$, and suppose its
  vertices are $x_1,\dots,x_n$ in a $D$-degeneracy order. We embed $H$ into
  $\Gamma$ one vertex at a time, in order, as follows.  Let $\psi_0$ be the
  \emph{trivial partial embedding} of no vertices of $H$ into $\Gamma$. Now for
  each $1\le i\le n$ in succession, we construct a \emph{partial embedding}
  $\psi_i$ of $\{x_1,\dots,x_i\}$ into $\Gamma$ as follows. We let $k$ be such
  that $\Delta_{k+1}<\deg(x_i)\le\Delta_k$, where $\Delta_{N+1}=0$. Denote by
  $N^-(x_i)=\{y_1,\ldots,y_{\ell}\}$ the \emph{back-neighbours} of~$x_i$, that
  is, the neighbours of $x_i$ preceeding $x_i$ in~$H$.  We choose $j$ minimal
  such that the vertices $\psi_{i-1}(y_1),\dots,\psi_{i-1}(y_\ell)$ have at
  least one common neighbour $v$ in $W_{k,j}\setminus\im\psi_{i-1}$. We define
  $\psi_i=\psi_{i-1}\cup\{x_i\to v\}$. If this is not possible, we say $\psi_i$
  (and the subsequent partial embeddings) do not exist and that the embedding
  strategy \emph{fails}.
\end{definition}

Note that this embedding strategy maintains that each $\psi_i$ which exists is
injective.  In order to prove that this embedding strategy does not fail we
must ensure that there exists~$j$ such that the images of already embedded
back-neighbours $\psi_{i-1}(y_1),\dots,\psi_{i-1}(y_\ell)$ have sufficiently
many common neighbours in $W_{k,j}\setminus\im\psi_{i-1}$. For this it will be
useful if we can maintain that during our embedding no $W_{k,j}$ gets filled up
too much. The following definition makes this precise, where we collect (some
of) our sets of embedded back-neighbours in a multiset~$\mathcal{B}$. This
multiset then has to satisfy certain conditions (given in~\ref{eq:Nice1}
and~\ref{eq:Nice2}) by our degeneracy condition on~$H$ and the given embedding
strategy, and we would like that the multisets do not fill up any sub-block
(this is~\ref{eq:Nice3}).

\begin{definition}[Well-behaved collection of embedded back-neighbours]
\label{def:well}
  Let $\Gamma(n,D)$, the partitions $W = \bigcup_{k} W_k$, and $W_k =
  \bigcup_{j} W_{k,j}$, and $\Delta_k$ be as in \Cref{construction}.  For
  $1\le t \le n,$ let $\mathcal{B}$ be a multiset $\{B_i\}_{i=1}^t,$ where each
  $B_i\se W$. Then $\cB$ is called {\em well-behaved} if
  \begin{enumerate}[label=\itmarab{NB}]
  \item  \label{eq:Nice1}
    $|B_i| \le D$  for all $1\le i \le t$, 
  \item \label{eq:Nice2} 
    for all $1\le k\le N$ and for all $u\in W_k$ we have $\big|\{ i \in [t] : u \in B_i \} \big| \le \Delta_k$,  and 
  \item \label{eq:Nice3}
    for each $1\le k\le N$ and each $1\le j\le\log n$, we have $\big|\bigcup \mathcal{B}\cap W_{k,j}\big|\le\tfrac12|W_{k,j}|$.
  \end{enumerate} 
\end{definition} 

For a set $B\se V(G),$ we denote by $N_G(B)$ the \emph{common neighbourhood} of
$B$ in $G,$ and omit the subscript when the graph $G$ is clear from context. The
next lemma shows that when we have a well-behaved collection~$\cB$ of embedded back-neighbours
then for any fixed vertex $u$ in our random block model that is not occupied by a vertex
from the collection, with some reasonable probability there is some $B\in\cB$ that is
entirely in the neighbourhood of~$u$ (which means that $u$ can be used for
embedding a vertex with back-neighbourhood embedded to~$B$).  The heart of the
proof of this lemma is a second moment calculation, needed for the application of a
special case of the Paley--Zygmund inequality.
We remark that we do not need Property~\ref{eq:Nice3} from \Cref{def:well} in this proof.

\begin{lemma}\label{lem:prob}
  Let $D, n$ be non-negative integers such that $n$ is sufficiently large, and
  let $1\le t\le n$. Let $\Gamma \sim \Gamma(n,D)$ be an instance of the random
  block model with vertex set~$W$ and let $\mathcal{B}$ be a well-behaved
  multiset of~$t$ subsets of $W$.  Fix $1\le k\le N$ and any $u\in
  W_k\setminus\bigcup\mathcal{B}$, and let $\cE$ be the event that there
  exists $B\in\mathcal{B}$ such that $u\in N_{\Gamma}(B)$. Then
  \[\Pr(\cE)\ge\min\Big\{\frac14,\, tn^{D^{1-k}-1}(\log n)^2(\log\log n)^D\Big\}\,.\]
\end{lemma}

\begin{proof}
  With each $B\in\cB$ we associate an \emph{intersection pattern}
  $\wv=\wv(B)\in\{0,\ldots,D\}^N$, where $\wv_k=|B\cap W_k|$.  Without loss of
  generality, we can assume $\wv(B)_1=0$ for each $B\in\cB$. Indeed, suppose
  that this special case of \Cref{lem:prob} holds. Given any multiset
  $\mathcal{B}$ that satisfies the assumption of the lemma, define
  $\mathcal{B}'$ by removing from each $B\in \mathcal{B}$ all elements in
  $W_1$. Then, since $p_{1,k}=1$, we have that $u\in N(B)$ for some
  $B\in\mathcal{B}$ if and only if $u\in N(B')$ for the corresponding
  $B'\in\mathcal{B}'$, so that the conclusion of the lemma for $\mathcal{B}$
  follows from that for $\mathcal{B}'$. We can similarly assume that $k\ge2$, since
  the desired probability in the case $k=1$ is equal to $1$ by the definition of
  $p_{i,1}$. Hence, we assume from now on that $\wv(B)_1=0$ and that $k\ge 2$.

  We next would like to argue that we can further restrict ourselves to the case
  that all $B\in \cB$ have the same intersection pattern $\wv$. More precisely,
  we claim that, if we can prove
  \begin{equation}
  \label{eq:prob:special}
    \Pr(\cE)\ge\min\Big\{\frac14\,,\frac{t\;  (\log n)^{2}}{2D} n^{D^{1-k}-1}(\log\log n)^{3D}\Big\}\,,
  \end{equation}
  in the case that all intersection patterns are the same, then we are
  done. Indeed, observe that, since each $B$ has size at most $D$, an
  intersection pattern is determined by a list of $D$ symbols that can either be
  from $[N]$ or a blank symbol, so that there are at most $(N+1)^D\le (2\log\log
  n)^D$ different intersection patterns, where we use
  \Cref{lem:blockmodel}\ref{lem:blockmodel:N}. Letting $\wv$ be the most common
  intersection pattern and restricting to the subcollection $\cB'\se\cB$ of at
  least $t(2\log\log n)^{-D}$ sets with intersection pattern $\wv$, we get
  from~\eqref{eq:prob:special} that the probability that $u\in N(B')$ for some
  $B'\in\mathcal{B'}$ is at least
  \[
  \min\Big\{\frac14,\frac{t\; (\log n)^{2}}{(2\log\log n)^D \cdot 2D} n^{D^{1-k}-1}(\log\log n)^{3D}\Big\}
  \ge \min\Big\{\frac14,tn^{D^{1-k}-1}(\log n)^2(\log\log n)^D\Big\}\,,
  \]
  as desired. So, we assume from now on also that all intersection patterns are $\wv$.
  
  For each $B\in \cB$, let $X_B$ denote the indicator random variable for the event $B\se N(u)$, and let $X=\sum_{B\in\cB}X_B.$ Then 
  \begin{equation}\label{eq:chebyshev} \Pr(\cE)= \Pr( X > 0) \ge \frac{(\E X)^2}{\E X^2},\end{equation}
  by Chebyshev's inequality. (This is also a special case of the Paley--Zygmund inequality.)
  We first note that 
  \begin{equation}\label{eq:fstMom}
    \E X = \sum_{B\in\cB} \E X_B = t \prod_{1\le i \le N} p_{i,k}^{\wv_i}.
  \end{equation}
  To bound the second moment from above we observe that  
  \begin{align}\label{eq:secMom1}
    \E X^2 &= \sum_{B,B'\in\cB} \Pr(X_B = 1, X_{B'}=1) 
      = \sum_{B\in\cB} \Pr\big(u\in N(B)\big)\cdot \sum_{B'\in\cB} \Pr \big(u\in N(B'\sm B)\big),
  \end{align}
  since the events $ u\in N(B)$ and $u\in N(B'\sm B)$ are independent for all
  $B,B'\in\cB$. Here, $\Pr \big(u\in N(\emptyset)\big) =1$ by convention.

  For $B,B'\in\cB,$ let $\ell = \ell (B,B')$ be the maximal index $j$ such
  that $B\cap B'\cap W_{j}\neq \emptyset,$ and set $\ell = 0$ if no such $j$
  exists. We now fix $B\in\mathcal{B}$ and find an upper bound on
  $\sum_{B'\in\cB} \Pr (u\in N(B'\sm B))$. We split this sum up according to
  $\ell(B,B')$. For $\ell=0$, we have
  \begin{equation}\label{eq:prob:ell0}
    \sum_{\substack{B'\in\cB \\ \ell(B,B')=0}}\Pr (u\in N(B'\sm B))=\sum_{\substack{B'\in\cB \\ \ell(B,B')=0}}\Pr \big(u\in N(B')\big)\le\E X\,.
  \end{equation}
  Observe that $\ell\neq 1$ since $\wv_1=0$, hence it remains to consider $\ell\ge 2$.
  In this case
  \[\Pr(u\in N(B'\sm B)) \le \prod_{i > \ell(B,B')} p_{i,k}^{\wv_i}\]
  for all $B'\in \cB,$ by using the trivial upper bound 1 for all elements in $(B'\sm B)\cap \bigcup_{i \le \ell(B,B')} W_i$.
  Since for every $\ell\in\{2,\ldots,N\},$ there are at most $\wv_{\ell}
  \Delta_\ell$ sets $B'\in\cB$ such that $B\cap B'\cap W_{\ell}\neq
  \emptyset$,
  we obtain 
  \begin{equation*}
    \sum_{\substack{B'\in\cB \\ \ell(B,B')>0}} \Pr (u\in N(B'\sm B)) 
    \le \sum_{\ell\in\{2,\ldots,N\}} \wv_{\ell} \Delta_\ell \prod_{i > \ell} p_{i, k}^{\wv_i}\,.
  \end{equation*}
  Putting this together with~\eqref{eq:prob:ell0}, we get
  \begin{align}\label{eq:secMom2}
    \sum_{\substack{B'\in\cB}} \Pr (u\in N(B'\sm B)) 
    &\le 2\max\Big\{\E X,\sum_{\ell\in\{2,\ldots,N\}} \wv_{\ell} \Delta_\ell \prod_{i > \ell} p_{i, k}^{\wv_i}\Big\}. 
  \end{align}
  Note that the right-hand side of~\eqref{eq:secMom2} does not depend on~$B$. We
  thus obtain from~\eqref{eq:secMom1} that
  \begin{align}
    \nonumber
    \E X^2 &\le 2\max\Big\{\E X,\sum_{\ell\in\{2,\ldots,N\}} \wv_{\ell} \Delta_\ell \prod_{i > \ell} p_{i, k}^{\wv_i}\Big\}\cdot \sum_{B\in\cB} \Pr(u\in N(B) )\\
    &\le \max\bigg\{2(\E X)^2,2\;\E X\cdot  D \max_\ell \Big\{\Delta_\ell \prod_{i > \ell} p_{i, k}^{\wv_i}\Big\}\bigg\}\,,\label{eq:secMom3}
  \end{align}
  where we use that $\sum_{\ell= 1}^N \wv_{\ell} \le D.$ Letting $\tilde{\ell}$
  be the index $\ell\in\{2,\dots,N\}$ maximising the expression $\Delta_\ell
  \prod_{i > \ell} p_{i, k}^{\wv_i},$ we obtain from~\eqref{eq:fstMom}
  and~\eqref{eq:secMom3} that
  \begin{align*}\label{eq:mostCommonPattern}
    \frac{(\E X)^2}{\E X^2} &\ge \min\Big\{\frac14,\frac{t  \prod_{1\le i \le \tilde{\ell} } p_{i,k}^{\wv_i}}{2D \Delta_{\tilde{\ell}}}	\Big\}
    \ge \min\Big\{\frac14,\frac{t \; p_{{\tilde \ell},k} ^D}{2D \Delta_{\tilde{\ell}}}\Big\}\\
    &\ge\min\Big\{\frac14,\frac{t\;  (\log n)^{2}}{2D} n^{D^{1-k}-1}(\log\log n)^{3D}\Big\}\,,
  \end{align*}	
  where the second inequality uses $p_{i,k}\ge p_{\tilde\ell,k}$ since
  $i\ge\tilde\ell$ and the third uses $\tilde\ell,k\ge2$ and \Cref{lem:blockmodel}\ref{lem:blockmodel:p} to substitute
  $p_{\tilde\ell,k}=n^{-D^{-1}+D^{-\tilde\ell}+D^{-k}}(\log^{2/D} n)(\log\log n)^3$,
  and $\Delta_{\tilde\ell}=n^{D^{1-\tilde\ell}}$. This together
  with~\eqref{eq:chebyshev} gives~\eqref{eq:prob:special} as required.
\end{proof}

We will now use \Cref{lem:prob} to show that $\Gamma\sim\Gamma(n,D)$ a.a.s.~has
the following pseudo-randomness property, which along with the bound on the
number of edges from \Cref{lem:blockmodel}\ref{lem:blockmodel:Gamma} is all
that we shall need of $\Gamma(n,D)$ to establish universality: For every
well-behaved multiset~$\cB$ in~$\Gamma$, in every subblock $W_{k,j}$ we have
many vertices~$u$ such that~$u$ is in the common neighbourhood of some member
of~$\cB$.

\begin{lemma}\label{lem:detprop}
  Given $D\ge 2$, the random block model $\Gamma(n,D)$ with vertex set
  $W=W_{1}\dcup\ldots\dcup W_{N}$ and sub-blocks $W_{k,1}\dcup\dots\dcup
  W_{k,\log n}=W_k$ as in \Cref{construction} a.a.s.\ satisfies the following.
  For every $1\le t\le n$, for every well-behaved multiset $\mathcal{B}$ of $t$
  subsets of $W$, for every $1\le k\le N$ and every $1\le j\le\log n$, we have
 \[\Big|\big\{u\in W_{k,j}\,:\,\exists B\in\mathcal{B}\,\text{with}\,u\in N(B)\big\}\Big|\ge\min\Big\{\frac{1}{16},\frac{t}{4}n^{D^{1-k}-1} (\log n)^2(\log\log n)^D\Big\}\,|W_{k,j}|.\]
\end{lemma}
\begin{proof}
  The assertion obviously is true for $k=1$, for all choices of $t$, $\cB$ and
  $j$, since vertices in $W_{1,j}$ have full degree and hence the number of
  $u\in W_{1,j}$ with $u\in N(B)$ for some~$B$ is
  $|W_{1,j}|\ge\frac1{16}|W_{1,j}|$. Therefore, we may assume $k\ge2$ in the following.

  We next want to argue that we can assume that $t$ is sufficiently small so that 
  \begin{equation}\label{eq:smallt}
    \frac{t}{4}n^{D^{1-k}-1} (\log n)^2(\log\log n)^D<\frac{1}{8}\,.
  \end{equation}
  More precisely, we argue that if we can show that in this case
  a.a.s.\ over all choices of $k\ge 2$, $j$, and $t$ satisfying~\eqref{eq:smallt} and over all choices of well-behaved $\mathcal{B}$ we have the
  bound
  \begin{equation}\label{eq:speccase}
    \big|\big\{u\in W_{k,j}\,:\,\exists B\in\mathcal{B}\,\text{with}\,u\in N(B)\big\}\big|\ge \frac{t}{4}n^{D^{1-k}-1}(\log n)^2(\log\log n)^D|W_{k,j}|\,,
  \end{equation}
  then this implies the lemma. Indeed, suppose that
  $\frac{t}{4}n^{D^{1-k}-1}(\log n)^2(\log\log n)^D\ge \frac{1}{8}$. Then we choose
  an integer $t'\le t$ such that
  \[\frac{1}{8}\ge \frac{t'}{4}n^{D^{1-k}-1}(\log n)^2(\log\log n)^D\ge\frac{1}{16}\,,\]
  which exists since $\tfrac14n^{D^{1-k}-1}(\log^2 n)(\log\log n)^D$ tends to
  zero as $n\to\infty,$ by our assumption $k\ge 2$.  Given $k$, $j$ and a
  well-behaved multiset $\mathcal{B}$ of $t$ subsets of $W$, we define a
  multiset $\mathcal{B}'$ by taking some $t'$ sets from $\mathcal{B}$. Trivially
  $\mathcal{B}'$ is well-behaved, so since~\eqref{eq:speccase} holds, $\{w\in
  W_{k,j}\,:\,\exists B'\in\mathcal{B}'\,,\, w\in N(B)\}$ has size at
  least \[\frac{t'}{4}n^{D^{1-k}-1}(\log n)^2(\log\log n)^D|W_{k,j}|\ge
  \frac1{16}|W_{k,j|}\] by choice of $t'$, and this set is a subset of the
  desired one. Hence, we may assume~\eqref{eq:smallt} from now on.
  
  Now fix $t\in [n]$ satisfying~\eqref{eq:smallt}, fix $2\le k \le N$,
  $j\in[\log n]$, and a well-behaved multiset $\mathcal{B}$ of $t$ subsets of
  $W$.  Since $\cB$ is well-behaved, it follows from~\ref{eq:Nice3} that there
  is a subset $U$ of $W_{k,j}$ of size $\tfrac12|W_{k,j}|$ which is disjoint
  from $\bigcup\mathcal{B}$. Fix such a set~$U$.  For each $w\in U,$ let $Y_{w}$
  denote the indicator random variable for the event that there exists $B\in\mathcal{B}$ with $w
  \in N(B)$. Observe that the variables~$Y_w$ are identically distributed and
  independent as $w$ ranges over $U$ and that~$Y_w$ is one with probability
  at least $t n^{D^{1-k}-1}(\log n)^{2}(\log\log n)^D$ by \Cref{lem:prob}
  and condition~\eqref{eq:smallt} on $t$.
  
  Now, the left-hand side of~\eqref{eq:speccase} is at least
  $Y=Y(k,j,\cB):=\sum_{{w\in U}} Y_{w}$.
  Using $|U|=\frac12|W_{k,j}|$, $|W_{k,j}|\ge\frac{1}{2\log n}|W_k|$, and $|W_k|=100\cdot 3^Dn^{1-D^{-k}}$,
  we conclude that
  \begin{equation}\begin{split}
    \label{eq:expnbs}
    \E Y &\ge \tfrac12|W_{k,j}|\cdot t n^{D^{1-k}-1}(\log n)^2(\log\log n)^D \\
    &\ge \tfrac1{4\log n}\cdot 100\cdot 3^D n^{1-D^{-k}}\cdot t n^{D^{1-k}-1}(\log n)^2(\log\log n)^D \\
    &> 2tn^{D^{1-k}-D^{-k}}(\log n)(\log\log n)^D
    \ge 2t(\log n)(\log\log n)^D\,.
  \end{split}\end{equation}
  Now, if~\eqref{eq:speccase} fails to hold for our fixed choice of $2\le k \le
  N$, $j\in[\log n]$, and well-behaved multiset $\mathcal{B}$ of size $t,$ then
  $Y < \E\; Y/2$, which occurs with probability at most
  \[2 \exp\big(-\tfrac{1}{12}\cdot 2t(\log n)(\log\log n)^D\big)=2\, n^{-\tfrac16 t (\log\log n)^D}\,,\]
  by \Cref{thm:Chernoff}, with $\delta=\tfrac12$.

  For our fixed $t$, we now take a union bound over the choices of
  $\mathcal{B}$, $k$, and $j$. Observe that $\mathcal{B}$ is given by a list of
  $tD$ vertices of $\Gamma(n,D)$ (together with null symbols to fill up sets of
  size smaller than $D$). Since $\Gamma(n,D)$ has at most $200\cdot 3^Dn$
  vertices by \Cref{lem:blockmodel}\ref{lem:blockmodel:Gamma}, we conclude that
  the number of choices of $\mathcal{B}$ is at most $\big(200\cdot
  3^Dn+1\big)^{tD}\le n^{2tD}$ for large~$n$. By
  \Cref{lem:blockmodel}\ref{lem:blockmodel:N} there are at most $2\log\log n$
  choices for~$k$, and by the definition of the sub-blocks, there are $\log n$
  choices for~$j$.  So the probability that there are $k,j$, and well-behaved
  $\cB$ of size $t$ for which~\eqref{eq:speccase} fails to hold is at most
  \[2 n^{2tD}(\log n)(2\log\log n) \cdot n^{-\tfrac16 t (\log\log n)^D} <n^{-2}\,,\]
  where the inequality holds for all sufficiently large $n$ since $(\log\log
  n)^D$ tends to infinity.

  Finally, we also take a union bound over the at most $n$ choices of $t$ to
  complete the proof that~\eqref{eq:speccase} holds a.a.s.~over all choices of
  $k\ge 2$, $j$, $t$ satisfying~\eqref{eq:smallt} and well-behaved $\mathcal{B}$
  of size~$t$ as desired.
\end{proof}

Our proof of \Cref{thm:Construction} now follows a strategy of Nenadov~\cite{rajko-thesis} which proceeds as follows.
We take
$\Gamma\sim\Gamma(n,D)$ satisfying the good property of
\Cref{lem:blockmodel}\ref{lem:blockmodel:Gamma} and the pseudorandomness
property of \Cref{lem:detprop}. We then fix a $D$-degenerate graph~$H$ we want
to embed. For this we use our embedding strategy, and we show inductively that
back-neighbourhoods are well-behaved, and that thus we can use the
pseudorandomness property to conclude that we will never fill up any sub-block
too much, and that this in turn implies that we can embed the next vertex in
some suitable sub-block.

\begin{proof}[Proof of \Cref{thm:Construction}]
  Let $D$ be fixed, let $n$ be large enough and let $\Gamma\sim \Gamma(n,D)$ be such that 
  $\Gamma$ has
  \begin{equation}\label{eq:edges}
    e(\Gamma)\le 80000n^{2-1/D}(\log n)^{2/D}(2\log\log n)^5
  \end{equation}
  edges, and such that for every $1\le t\le n$, every well-behaved multiset $\mathcal{B}$ of $t$ subsets of $W$, for every $1\le k\le N$ and every $1\le j\le\log n$, we have
  \begin{multline}\label{eq:pseudo}
    \big|\big\{u\in W_{k,j}\,:\,\exists B\in\mathcal{B}\,\text{with}\,u\in N(B)\big\}\big| \\
    \ge\min\Big\{\frac{1}{16},\frac{t}{4}n^{D^{1-k}-1} (\log n)^2(\log\log n)^D\Big\}\,|W_{k,j}|\,.\qquad
  \end{multline}
  By \Cref{lem:blockmodel}\ref{lem:blockmodel:Gamma} and \Cref{lem:detprop}
  the properties~\eqref{eq:edges} and~\eqref{eq:pseudo} occur a.a.s.  We will
  show that these imply that $\Gamma$ is universal for $D$-degenerate graphs on
  $n$ vertices.
 
  Recall from \Cref{construction} that, for every $1\le k \le N,$ the block
  $W_k$ is of size $100\cdot 3^Dn^{1-D^{-k}}$ and is partitioned into sub-blocks
  $W_{k,j}$ of sizes
  \begin{equation}\label{eq:sub-blocks}
    |W_{k,1}|=\frac12|W_{k}|=50\cdot 3^Dn^{1-D^{-k}}
    \quad\text{and}\quad
    |W_{k,j}|=\frac1{2\log n}|W_{k}|=\frac{50}{\log n}\cdot 3^Dn^{1-D^{-k}}\,
  \end{equation}
  for $2\le j\le \log n$, respectively, and that $\Delta_k = n^{D^{1-k}}$.  Fix a $D$-degenerate graph
  $H$ on $n$ vertices, suppose its vertices are $x_1,\dots,x_n$ in a
  $D$-degeneracy order, and run the embedding strategy as given
  in \Cref{strategy}. Let $\psi_i$ be the partial embedding of
  $\{x_1,\ldots,x_i\}$.

  We next recursively define numbers $L_{k,j}$ for $1\le k\le N$ and $1\le j\le\log n$ as follows:
  \[L_{k,j}=
  \begin{cases}
    \frac{2nD}{\Delta_{k+1}}=2n^{1-D^{-k}} \quad&\text{if $1\le k\le N-1$ and $j=1$} \\
    n & \text{if $k=N$ and $j=1$ \,} \\
    \frac1{4\log n} L_{k,j-1} & \text{if $1\le k\le N$ and $j>1$ \,.}
  \end{cases}
  \]
  We shall show that $L_{k,j}$ is an upper bound on the number of vertices our embedding strategy uses in~$W_{k,j}$.
  Before turning to this, observe that
  \begin{equation}\label{eq:Lbound}
    L_{k,j}+1\le\tfrac1{16}|W_{k,j}| \qquad\text{and}\qquad L_{k,\log n}<\log n
  \end{equation}
  for each $k,j$. Indeed, the second inequality holds with lots of room to spare:
  \[
   L_{k,\log n}=\frac{1}{(4\log n)^{\log n}}L_{k,1}\le\frac{n}{(4\log n)^{\log n}}<1\,,
  \]
  where the final inequality is since $4\log n>e$ and $e^{\log n}=n$.
  To see the first inequality in~\eqref{eq:Lbound}, note that we have
  $L_{N,1}=n$ and $|W_{N,1}|\ge\frac12 \cdot 100n$ by
  \Cref{lem:blockmodel}\ref{lem:blockmodel:W}, and for $k<N$ we have
  $L_{k,1}=2n^{1-D^{-k}}$ and $|W_{k,1}|=50\cdot 3^Dn^{1-D^{-k}}$
  by~\eqref{eq:sub-blocks}. Similarly, for $j\ge 2$ we have
  $L_{N,j}\le\frac{n}{4\log n}$ and $|W_{N,j}|\ge\frac1{2\log n} \cdot 100n$ by
  \Cref{lem:blockmodel}\ref{lem:blockmodel:W}, and for $k<N$ we have
  $L_{k,j}\le \frac{1}{4\log n} 2n^{1-D^{-k}}$ and $|W_{k,j}|=\frac{50}{\log n}
  \cdot 3^Dn^{1-D^{-k}}$ by~\eqref{eq:sub-blocks}.
  
  Now, for any step $1\le i\le n$ in our embedding strategy, consider the property 
  \begin{align*}
    &P(i) \,: \quad \text{$\psi_{i}$ exists and }  |W_{k,j}\cap\im\psi_i|\le L_{k,j} \text{ for all $1\le k\le N$ and all $1\le j\le \log n$}\,. 
  \end{align*}
  The property $P(n)$ implies that $\Gamma$ contains $H$ as a subgraph,
  finishing the proof of our theorem. We shall prove that $P(i)$ holds for all $1\le
  i\le n$ inductively.  Consider $i=1$ first. Since $x_{1}$ has no
  back-neighbours it can be embedded arbitrarily in $W_{k,1},$ where $k$ is
  determined by $\deg_H(x_1).$ Then $|W_{k',j'}\cap\im\psi_i|\in\{0,1\}$ for all
  $(k',j'),$ and thus $P(1)$ holds trivially.

  Let now $i>1$ and assume that $P(i-1)$ holds. Let $k$ be minimal such that
  $\deg_H(x_i)\le \Delta_k.$ We will first show that it is possible to embed
  $x_i$ into $W_{k,\log n}$, and thus the embedding of~$x_i$ succeeds, and then
  that $P(i)$ holds inductively.
  Indeed, let $y_1,\ldots,y_{\ell}$ be the at most~$D$ neighbours
  of $x_i$ in $\{x_1,\ldots, x_{i-1}\}$ that are already embedded in $\Gamma$ by
  $\psi_{i-1}$, and let $B=\{\psi_{i-1}(y_1),\ldots,\psi_{i-1}(y_\ell)\}.$ Note
  that clearly $\{B\}$ is well-behaved. 
  Thus, by~\eqref{eq:pseudo},
  the number of vertices~$u$ in $W_{k,\log n}\cap N(B)$ is at least
  \begin{equation*}\begin{split}
  \frac{1}{4}n^{D^{1-k}-1} (\log n)^2 & (\log\log n)^D|W_{k,j}| \\
  & \eqByRef{eq:sub-blocks} \frac{1}{4}n^{D^{1-k}-1} (\log n)^2(\log\log n)^D\frac{50}{\log n}\cdot 3^Dn^{1-D^{-k}} \\
  & =\frac{25}{2}\cdot 3^D\log n(\log\log n)^D n^{D^{-k}(D-1)}\ge 2\log n\,.
  \end{split}\end{equation*}
  Using $P(i-1)$ and~\eqref{eq:Lbound}, at most $L_{k,\log n}<\log n$ of these
  are in the image of $\psi_{i-1}$, so that we can choose an image for $x_i$. In
  particular, there exists a minimal $j$ such that $\big(W_{k,j}\cap N(B)\big)
  \sm \im \psi_{i-1}\neq \emptyset$ and therefore the embedding strategy
  succeeds at step $i$. For the following argument, we fix this $j$.

  To finish the induction step, assume for a contradiction that $P(i)$ fails to
  hold. Since only $|W_{k,j}\cap\im\psi_i|$ changes in step~$i$, this implies
  that 
  \begin{alignat}{2}\label{eq:indstep1}
    L_{k,j}< |W_{k,j}\cap\im\psi_i|&\le L_{k,j}+1\,, \quad&&\text{and} \\
    \label{eq:indstep2}
    |W_{k',j'}\cap\im\psi_i|&\le L_{k',j'} &&\text{for all $(k',j')\neq (k,j)$}\,.
  \end{alignat}
  First assume that $j=1$.  If
  also $k=N$, then the fact that $\im\psi_i$ has size $i\le n=L_{1,N}$
  immediately contradicts~\eqref{eq:indstep1}. If $k<N$ on
  the other hand, then recall that all vertices in $\psi_i^{-1}(W_k)$ have
  degree at least $\Delta_{k+1}$ by our embedding strategy; by
  \Cref{degree-observation}, there are at most $2Dn/\Delta_{k+1} = L_{k,1}$
  such vertices in $H$, again contradicting~\eqref{eq:indstep1}.

  Hence, it remains to consider the case $j\ge 2$.
  We construct $\mathcal{B}$ as follows: For each $x\in V(H)$ with $\psi_i(x)\in
  W_{k,j}$, we add the set $B_x=\psi_i\big(N^-_H(x)\big)$ to
  $\mathcal{B}$. Observe that this is a multiset since some vertices of $H$ may
  have identical back-neighbourhoods, and $|\cB|>L_{k,j}$
  by~\eqref{eq:indstep1}. We claim that $\mathcal{B}$ is well-behaved. Indeed,
  $|B|\le D$ for all $B\in\cB$ since we embed vertices of $H$ in the
  $D$-degeneracy order. Next, we verify~\ref{eq:Nice2}: Given $u\in W_{k'}$ for
  some $1\le k'\le N$, the number of sets $B$ of $\mathcal{B}$ containing $u$ is
  zero if $u\not\in\im\psi_i$. If $u\in\im\psi_i$, then $u\in B_x$ only if
  $\psi_i^{-1}(u)$ is a neighbour of $x$ in $H$. Since $u\in W_{k'}$, the degree
  of $\psi_i^{-1}(u)$ is at most $\Delta_{k'}$ and hence there are at most
  $\Delta_{k'}$ choices of $x$ such that $u\in B_x$, giving~\ref{eq:Nice2}.
  Finally, for verifying~\ref{eq:Nice3}, note that since
  $\bigcup\mathcal{B}\subset\im\psi_i$, the number of vertices of
  $\bigcup\mathcal{B}$ in any given $W_{k',j'}$ is at most $L_{k',j'}+1$ (with
  equality only for $k,j$) by~\eqref{eq:indstep1} and~\eqref{eq:indstep2} and
  $L_{k',j'}+1<\tfrac1{16}|W_{k',j'}|$ by~\eqref{eq:Lbound}. This finishes the
  check that $\mathcal{B}$ is well-behaved.
 
  Hence, it follows from~\eqref{eq:pseudo} that 
  \begin{multline}\label{eq:filled}
    \big|\big\{u\in W_{k,j-1}\,:\,\exists B\in\mathcal{B}\,\text{with}\,u\in N(B)\big\}\big| \\
    \ge\min\Big\{\frac{1}{16},\frac{L_{k,j}}{4}n^{D^{1-k}-1} (\log n)^2(\log\log n)^D\Big\}\,|W_{k,j-1}|\,.
    \qquad
  \end{multline}
  Moreover, we have
  \begin{multline*}
    \frac{L_{k,j}}{4}n^{D^{1-k}-1} (\log n)^2(\log\log n)^D|W_{k,j-1}| \geByRef{eq:sub-blocks}
    L_{k,j}3^Dn^{D^{-k}(D-1)} \log n(\log\log n)^D \\
    =\frac{L_{k,j-1}}{4\log n}3^Dn^{D^{-k}(D-1)} \log n(\log\log n)^D
    > L_{k,j-1} +1\,,
  \end{multline*}
  where the equality uses the definition of $L_{k,j}$. Combining this
  with~\eqref{eq:Lbound}, we obtain that the right hand side
  of~\eqref{eq:filled} is strictly larger than $L_{k,j-1}$, which is an upper
  bound for $|W_{k,j-1}\cap\im\psi_i|$. It follows that there is some $x$ with
  $\psi_i(x)\in W_{k,j}$ such that $N(B_x)$ contains a vertex of $W_{k,j-1}$
  outside $\im\psi_i$. But this is a contradiction: we could have embedded $x$
  to $W_{k,j-1}$ and therefore would not have embedded it to $W_{k,j}$. This
  proves $P(i),$ and thus, by induction, $P(n)$ as desired.
\end{proof}

\section{Concluding remarks} 

In this paper, we initiated the study of $\cH$-universal graphs, when $\cH=\cH(n,D)$ is the class of all $n$-vertex $D$-degenerate graphs. We determined that the minimum number of edges of $\cH(n,D)$-universal graphs is $O(n^{2-1/D}(\log^{2/D} n)(\log\log n)^5)$.

The counting argument shows that any $\cH(n,D)$-universal graph has at least $\Omega(n^{2-1/D})$ edges, and it remains open whether a polylogarithmic factor is needed. While in our proof, the $(\log n)^{2/D}$-factor is needed, the $(\log\log n)^5$-factor may be shaved off with our proof strategy, albeit the proof becoming more technical. We do believe, however, that no polylog-factor should be necessary.

\begin{conjecture}
	The minimum number of edges of an $\cH(n,D)$-universal graph is $\Theta(n^{2-1/D}).$
\end{conjecture}

In this paper, we focused on minimising the number of edges of universal graphs. One may additionally ask for the minimum number of vertices of such a sparse universal graph. In~\cite{ackrrs2000}, one of the main motivations for considering the random graph $ G(n,p)$ was that it provides sparse graphs on $(1+\varepsilon)n$ vertices that are $\cH_{\Delta}(n)$-universal. The number of vertices of our constructed $\Gamma$ is between $100n$ and $200\cdot3^D n$, see \Cref{lem:blockmodel}. We believe that a similar construction, with a more careful analysis of the embedding scheme, will provide an $\cH(n,D)$-universal graph on $(1+\eps)n$ vertices with a similar number of edges. Roughly, one would need to choose $N$ slightly smaller such that the union $W_1\dcup\dots \dcup W_{N-1}$ has size about $\tfrac12\eps n$, set the size of $W_N$ to be $(1+\tfrac12\eps)n$, and adjust the probabilities $p(i,N)$ and the sizes of the subblocks of $W_N$. We see no reason why this should cause genuine difficulty (it does make for a rather more intricate optimisation problem), but did not check the details.

Finally, it would be interesting to determine the minimum number of edges an $n$-vertex $\cH(n,D)$-universal graph can have. We suspect that one could get an upper bound $\tilde{O}\big(n^{2-1/2D}\big)$ which Ferber and Nenadov~\cite{fn2018} proved for spanning $\cH_\Delta(n,D)$-universality by using something like the random block model, setting aside a large independent set of vertices of degree at most $2D$ in the embedding and finishing off with a matching argument to embed these (much as in~\cite{fn2018}). However, since we cannot ask for these set-aside vertices to be widely separated, making this argument work is likely to be harder. For this problem, it would already be interesting to improve on~\cite{fn2018} and show $o\big(n^{2-1/2D}\big)$ edges can suffice for spanning $\cH_\Delta(n,D)$-universality.

\section*{Acknowlegement} The authors would like to thank the organisers and sponsors of the Second Armenian Workshop on Graphs, Combinatorics, Probability, where we started this project, for their hospitality.

\bibliographystyle{plain}
\bibliography{references}

\end{document}